\documentclass[12pt,abstract]{scrartcl}
\usepackage[dvipdfmx]{graphicx}
\usepackage{amsmath,amsthm}
\usepackage[round]{natbib}
\newtheorem{proposition}{Proposition}
\newtheorem{lemma}{Lemma}

\usepackage{newtxtext,newtxmath}
\usepackage{xcolor}

\begin{document}
\title{Advertising in an oligopoly with differentiated goods under general demand and cost functions: A differential game approach}

\author{Masahiko Hattori\thanks{%
mhattori@mail.doshisha.ac.jp} \\
Faculty of Economics, Hokkai-Gakuen University,\\
[.02cm] Toyohira-ku, Sapporo, Hokkaido, 062-8605, Japan,\\
[.01cm] \textrm{and} \\
[.1cm] Yasuhito Tanaka\thanks{%
yatanaka@mail.doshisha.ac.jp}\\
[.01cm] Faculty of Economics, Doshisha University,\\
Kamigyo-ku, Kyoto, 602-8580, Japan.\\
}

\date{}
\maketitle

\begin{abstract}
We present an analysis of advertising activities in a dynamic oligopoly with differentiated goods by differential game approach under general demand and cost functions. Mainly we show the following results. The comparison of the open-loop solution and that of the closed-loop solution depends on whether the outputs of the firms are strategic substitutes or strategic complements, and the memoryless closed-loop solution and the feedback solution are equivalent when there is no spillover effect of advertising activities.
\end{abstract}

\begin{description}
\item[Keywords:] oligopoly; advertising; differential game; general demand and cost functions
\end{description}

\begin{description}
\item[JEL Classification No.:] C73, D43, L13.
\end{description}

\section{Introduction}

In this paper we present an analysis of advertising activities in a dynamic oligopoly model with differentiated goods by differential game approach. There are many studies of dynamic oligopoly by differential game theory, for example, \cite{cl0}, \cite{cl6}, \cite{cl1}, \cite{cl3}, \cite{cl2}, \cite{cl4}, \cite{fu}, \cite{fu1}, and \cite{lam18}. Among them \cite{cl6} analyzed the advertising activities in a Cournot oligopoly with differentiated goods. However, most of these studies including \cite{cl6} used a model of linear demand functions and quadratic or linear cost functions. These assumptions are very limited. We study the problem addressed by them in an oligopoly under general demand and cost functions.

In the next section we present a model and assumptions. In Section 3 we consider the open-loop solution of advertising activities. In Section 4  we present the memoryless closed-loop solution of advertising activities. \cite{cl6} claim that the open-loop solution and the memoryless closed-loop solution coincide.  \cite{cl6} said
\begin{quote}
the optimal values of the control variables of each player are not affected by state variables different from its own market size.
\end{quote}
However, it is incorrect. The market size (or the accumulated advertising effect) of one good affects its price, then the demands of other goods are affected. We show that the comparison of the open-loop solution and the closed-loop solution depends on whether the outputs of the firms are strategic substitutes or strategic complements. If the outputs of the firms are strategic substitutes, the steady state value of the accumulated effects of advertising in the closed-loop solution is smaller than that in the open-loop solution. If the outputs of the firms are strategic complements, the steady state value of the accumulated effects of advertising in the closed-loop solution is larger than that in the open-loop solution. In Section 5 we analyze the feedback solution using Hamilton-Jacobi-Bellman equation, and show the equivalence of the memoryless closed-loop solution and the feedback solution when there is no spillover effect of advertising activities. In Section 6 we consider advertising activities in a cartel. We show that the adversing investment in the cartel may be larger than that in the open-loop solution. This is due to the spillover of advertising activities.

\section{The model}

Consider an oligopoly with $n$ firms in which at any $t\in [0, \infty)$  they produce differentiated goods to maximize their discounted profits. The firms are called Firms 1, 2, $\dots$, $n$. Let $q_i(t)$ be the output of Firm $i,\ i\in \{1, 2, \dots, n\}$, $p_i(t)$ be the price at $t$. The inverse demand function is written as
\[p_i(t)=p_i(A_i(t),q_1(t),q_2(t),\dots,q_n(t)),\ i\in \{1, 2, \dots, n\}.\]
$A_i(t)$ represents the market size of the good of Firm $i$, or the accumulated effects of advertising for the good of Firm $i$. It is a state variable. Denote $p_i(A_i(t),q_1(t),q_2(t),\dots,q_n(t))$ by $p_i$. The inverse demand functions for all firms are symmetric. We assume 
\begin{equation*}
\frac{\partial p_i(A_i(t),q_1(t),q_2(t),\dots,q_n(t))}{\partial q_i(t)}<0,\ \frac{\partial p_i(A_i(t),q_1(t),q_2(t),\dots,q_n(t))}{\partial q_j(t)}<0,\ j\neq i,
\end{equation*}
\[\frac{\partial p_i(A_i(t),q_1(t),q_2(t),\dots,q_n(t))}{\partial A_i(t)}>0,\]
and

\begin{equation}
\frac{\partial p_i}{\partial A_i(t)}+\frac{\partial^2 p_i}{\partial A_i(t)\partial q_i(t)}q_i(t)>0.\label{aq}
\end{equation}
If the outputs of the firms are strategic substitutes
\[\frac{\partial p_j}{\partial q_i(t)}+\frac{\partial^2 p_j}{\partial q_i(t)\partial q_j(t)}q_j(t)<0,\ j\neq i,\]
and if they are strategic complements
\[\frac{\partial p_j}{\partial q_i(t)}+\frac{\partial^2 p_j}{\partial q_i(t)\partial q_j(t)}q_j(t)>0,\ j\neq i.\]
The production cost of Firm $i$ is
\[c(q_i(t)),\ i\in \{1, 2, \dots, n\},\ c(q_i(t))>0.\]
 This is common to all firms. $c(q_i(t))$ is strictly increasing and convex, that is, $c'_i(q_i(t))>0$ and $c''_i(q_i(t))\geq 0$.

Let $k_i(t)$ be the advertising investment by Firm $i$. The moving of $A_i(t)$ is governed by 
\begin{equation}
\frac{dA_i(t)}{dt}=\Gamma(k_i(t),K_{-i}(t))-\delta A_i(t),\ i\in \{1, 2, \dots, n\},\ \Gamma(k_i(t),K_{-i}(t))>0,\label{m1}
\end{equation}
where
\[K_{-i}(t)=\sum_{j\neq i}k_j(t).\]
$\delta>0$ is the constant depreciation rate. We assume that $\Gamma(k_i(t),K_{-i}(t))$ is strictly increasing and concave, that is,
\begin{equation}
\left\{
\begin{split}
&\frac{\partial \Gamma(k_i(t),K_{-i}(t))}{\partial k_i(t)}>0,\ \frac{\partial \Gamma(k_i(t),K_{-i}(t))}{\partial K_{-i}(t)}>0,\ \frac{\partial^2 \Gamma(k_i(t),K_{-i}(t))}{\partial k_i(t)^2}\leq 0,\\
&\frac{\partial^2 \Gamma(k_i(t),K_{-i}(t))}{\partial k_i(t)^2}+(n-1)\frac{\partial^2 \Gamma(k_i(t),K_{-i}(t))}{\partial k_i(t)\partial K_{-i}(t)}\leq 0,\label{as0}\\
&\left|\frac{\partial \Gamma(k_i(t),K_{-i}(t))}{\partial k_i(t)}\right|>\left|\frac{\partial \Gamma(k_i(t),K_{-i}(t))}{\partial k_i(t)}\right|.
\end{split}
\right.
\end{equation}
The last condition means that the direct effect of advertising is larger than the spillover effect. $\Gamma$ is common to all firms. Denote $\Gamma(k_i(t),K_{-i}(t))$ by $\Gamma_i$.

The advertising investment cost of Firm $i$ is
\[\gamma(k_i(t)),\ i\in \{1, 2, \dots, n\},\ \gamma(k_i(t))>0.\]
We assume that it is strictly increasing and strictly convex, that is
\begin{equation}
\gamma'(k_i(t))>0,\ \gamma''(k_i(t))>0.\label{as2}
\end{equation}
$\gamma$ is common to all firms.

The instantaneous profit of Firm $i$ is written as
\[p_i(A_i(t),q_1(t), q_2), \dots, q_n(t))q_i(t)-c(q_i(t))-\gamma(k_i(t)).\]
The objective of Firm $i$ is 
\[\max_{q_i(t), k_i(t)}\int_{0}^{\infty}e^{-\rho t}[p_i(A_i(t),q_1(t), q_2), \dots, q_n(t))q_i(t)-c(q_i(t))-\gamma(k_i(t))]dt,\]
subject to (\ref{m1}). $\rho>0$ is the discount rate.

The present value Hamiltonian function for Firm $i,\ i\in \{1, 2, \dots, n\}$, is
\begin{align*}
\mathcal{H}_i(q_i(t), k_i(t))=&e^{-\rho t}\big\{p_i(A_i(t),q_1(t), q_2), \dots, q_n(t))q_i(t)-c(q_i(t))-\gamma(k_i(t))\\
&+\lambda_{ii}(t)(\Gamma_i(k_i(t), K_{-i}(t))-\delta A_i(t))+\sum_{j\neq i}\lambda_{ij}(t)(\Gamma_j(k_j(t), K_{-j}(t))-\delta A_j(t))\}.
\end{align*}
The current value Hamiltonian function for Firm $i,\ i\in \{1, 2, \dots, n\}$, is
\begin{align*}
\hat{\mathcal{H}}_i(q_i(t), k_i(t))=&e^{\rho t}\mathcal{H}_i=p_i(A_i(t),q_1(t), q_2), \dots, q_n(t))q_i(t)-c(q_i(t))-\gamma(k_i(t))\\
&+\lambda_{ii}(t)(\Gamma_i(k_i(t), K_{-i}(t))-\delta A_i(t))+\sum_{j\neq i}\lambda_{ij}(t)(\Gamma_j(k_j(t), K_{-j}(t))-\delta A_j(t)).
\end{align*}
Let
\[\mu_{ii}(t)=e^{-\rho t}\lambda_{ii}(t),\ \mu_{ij}(t)=e^{-\rho t}\lambda_{ij}(t).\]
$\mu_{ii}(t)$ and $\mu_{ij}(t)$ are the costate variables. Denote $\hat{\mathcal{H}}_i(q_i(t), k_i(t))$ by $\hat{\mathcal{H}}_i$. 

\section{Advertising in dynamic oligopoly: Open-loop solution}

We seek to the solution of the open-loop approach. 

The first order conditions for Firm $i$ are
\begin{equation}
\frac{\partial \hat{\mathcal{H}}_i}{\partial q_i(t)}=p_i+\frac{\partial p_i}{\partial q_i(t)}q_i(t)-c'(q_i(t))=0,\label{q}
\end{equation}
and
\begin{equation}
\frac{\partial \hat{\mathcal{H}}_i}{\partial k_i(t)}=-\gamma'(k_i(t))+\lambda_{ii}(t)\frac{\partial \Gamma_i}{\partial k_i(t)}+\sum_{j\neq i}\lambda_{ij}(t)\frac{\partial \Gamma_j}{\partial K_{-j}(t)}=0.\label{k}
\end{equation}
The second order condition about production is
\[\frac{\partial \hat{\mathcal{H}}^2_i}{\partial q_i(t)^2}=2\frac{\partial p_i}{\partial q_i(t)}+\frac{\partial^2 p_i}{\partial q_i(t)^2}q_i(t)-c''(q_i(t))<0.\]
The second order condition about advertising investment is
\[\frac{\partial^2 \hat{\mathcal{H}}_i}{\partial k_i(t)^2}=-\gamma''(k_i(t))+\lambda_{ii}(t)\frac{\partial^2 \Gamma_i}{\partial k_i(t)^2}+\sum_{j\neq i}\lambda_{ij}(t)\frac{\partial^2 \Gamma_j}{\partial K_{-j}(t)^2}<0,\ j\neq i.\]
The adjoint conditions are
\begin{equation}
-\frac{\partial \hat{\mathcal{H}}_i}{\partial A_i(t)}=\frac{\partial \lambda_{ii}(t)}{\partial t}-\rho \lambda_{ii}(t),\ i\in \{1, 2, \dots,n\},\label{a1}
\end{equation}
and
\begin{equation}
-\frac{\partial \hat{\mathcal{H}}_i}{\partial A_j(t)}=\frac{\partial \lambda_{ij}(t)}{\partial t}-\rho \lambda_{ij}(t),\ j\neq i.\label{a1a}
\end{equation}
We have
\begin{equation}
\frac{\partial \hat{\mathcal{H}}_i}{\partial A_i(t)}=\frac{\partial p_i}{\partial A_i(t)}q_i(t)-\delta\lambda_{ii}(t),\label{b1}
\end{equation}
and
\begin{equation}
\frac{\partial \hat{\mathcal{H}}_i}{\partial A_j(t)}=-\delta\lambda_{ij}(t),\ j\neq i.\label{b2}
\end{equation}
At the steady state
\begin{equation}
\frac{dA_i(t)}{dt}=\Gamma_i-\delta A_i(t)=0,\label{st1}
\end{equation}
\[\frac{\partial \lambda_{ii}(t)}{\partial t}=0,\ \frac{\partial \lambda_{ij}(t)}{\partial t}=0,\ i\in \{1, 2, \dots, n\},\ j\neq i.\]
By symmetry of the oligopoly we can assume $\lambda_{ii}(t)=\lambda_{jj}(t)$, $\lambda_{ij}(t)=\lambda_{jl}(t)=\lambda_{ji}(t)$,\ $j\neq i,\ l\neq i, j$, $k_i(t)=k_j(t)$, $q_i(t)=q_j(t)$, $A_i(t)=A_j(t)$ for $j\neq i$, and so on. Denote the steady state values of $q_i(t)$, $k_i(t)$, $A_i(t)$, $\lambda_{ii}(t)$ and $\lambda_{ij}(t)$ by $q^*$, $k^*$,$A^*$, $\lambda_{own}$ and $\lambda_{other}$. 

From (\ref{a1a}) and (\ref{b2}) we have
\[\lambda_{other}=0.\]
From (\ref{a1}) and (\ref{b1})
\[(\rho+\delta)\lambda_{own}=\frac{\partial p_i}{\partial A_i(t)}q^*.\]
(\ref{k}) is reduced to
\[-\gamma'(k^*)+\lambda_{own}\frac{\partial \Gamma_i}{\partial k_i(t)}=0.\]
Therefore, we get
\begin{equation}
\frac{\partial p_i}{\partial A_i(t)}q^*-(\rho+\delta)\frac{\gamma'(k^*)}{\frac{\partial \Gamma_i}{\partial k_i(t)}}=0.\label{ga1}
\end{equation}
From (\ref{m1}) and (\ref{a1}), under the symmetry condition, 
\begin{equation}
\frac{dA_i(t)}{dt}=\Gamma(k_i(t),(n-1)k_i(t))-\delta A_i(t),\label{d1}
\end{equation}
and
\begin{equation}
\frac{\partial \lambda_{ii}(t)}{\partial t}=-\frac{\partial p_i}{\partial A_i(t)}q_i(t)+(\rho+\delta) \lambda_{ii}(t).\label{d2}
\end{equation}
From (\ref{k}) with $\lambda_{ij}(t)=0$ \textcolor{black}{and $\lambda_{ii}(t)=\frac{\gamma'(k_i(t))}{\frac{\partial \Gamma_i}{\partial k_i(t)}}$}, under the symmetry condition,
\[\frac{\partial k_i(t)}{\partial \lambda_{ii}(t)}=\frac{\left(\frac{\partial \Gamma_i}{\partial k_i(t)}\right)^2}{\gamma''\frac{\partial \Gamma_i}{\partial k_i(t)}-\gamma'\left[\frac{\partial^2 \Gamma_i}{\partial k_i(t)^2}+(n-1)\frac{\partial^2 \Gamma_i}{\partial k_i(t)\partial K_{-i}(t)}\right]}.\]
Thus, we have
\[\frac{\partial \Gamma(k_i(t),(n-1)k_i(t))}{\partial \lambda_{ii}(t)}=\left(\frac{\partial \Gamma_i}{\partial k_i(t)}+(n-1)\frac{\partial \Gamma_i}{\partial K_{-i}(t)}\right)\frac{\left(\frac{\partial \Gamma_i}{\partial k_i(t)}\right)^2}{\gamma''\frac{\partial \Gamma_i}{\partial k_i(t)}-\gamma'\left[\frac{\partial^2 \Gamma_i}{\partial k_i(t)^2}+(n-1)\frac{\partial^2 \Gamma_i}{\partial k_i(t)\partial K_{-i}(t)}\right]}.\]

Consider a system of dynamic equations (\ref{d1}) and (\ref{d2}). Linearization of them around the steady state is as follows.
\[
\begin{bmatrix}
\frac{dA_i(t)}{dt}\\[.15cm]
\frac{\partial \lambda_{ii}(t)}{\partial t}
\end{bmatrix}
=\Omega
\begin{bmatrix}
A_i(t)-A^*\\
\lambda_{ii}(t)-\lambda^*
\end{bmatrix},
\]
\[
\Omega=
\begin{bmatrix}
-\delta&\left(\frac{\partial \Gamma_i}{\partial k_i(t)}+(n-1)\frac{\partial \Gamma_i}{\partial K_{-i}(t)}\right)\frac{\left(\frac{\partial \Gamma_i}{\partial k_i(t)}\right)^2}{\gamma''\frac{\partial \Gamma_i}{\partial k_i(t)}-\gamma'\left[\frac{\partial^2 \Gamma_i}{\partial k_i(t)^2}+(n-1)\frac{\partial^2 \Gamma_i}{\partial k_i(t)\partial K_{-i}(t)}\right]}\\
-\frac{\partial }{\partial A_i(t)}\left(\frac{\partial p_i}{\partial A_i(t)}q_i(t)\right)&\rho+\delta
\end{bmatrix}.
\]
The Jacobian matrix $\Omega$ has the following trace and determinant
\[\mathrm{tr}\ (\Omega)=\rho>0,\]
\begin{align*}
&\det{\Omega}\\
=&\left(\frac{\partial \Gamma_i}{\partial k_i(t)}+(n-1)\frac{\partial \Gamma_i}{\partial K_{-i}(t)}\right)\frac{\left(\frac{\partial \Gamma_i}{\partial k_i(t)}\right)^2}{\gamma''\frac{\partial \Gamma_i}{\partial k_i(t)}-\gamma'\left[\frac{\partial^2 \Gamma_i}{\partial k_i(t)^2}+(n-1)\frac{\partial^2 \Gamma_i}{\partial k_i(t)\partial K_{-i}(t)}\right]}\frac{\partial }{\partial A_i(t)}\left(\frac{\partial p_i}{\partial A_i(t)}q_i(t)\right)\\
&-\delta(\rho+\delta).
\end{align*}
For the steady state to be a saddle point we need
\[\det{\Omega}<0.\]
If $\det{\Omega}>0$, the steady state is unstable. From $\frac{dA_i(t)}{dt}=0$, we have
\[\Gamma(k^*,(n-1)k^*)=\delta A^*.\]
Thus,
\[\frac{dk^*}{dA^*}=\frac{\delta}{\frac{\partial \Gamma_i}{\partial k_i(t)}+(n-1)\frac{\partial \Gamma_i}{\partial K_{-i}(t)}}.\]
 Let $\Phi$ be the left-hand side of (\ref{ga1}). Differentiating $\Phi$ with respect to $A_i(t)$ at the steady state yields 
\begin{align*}
\frac{\partial \Phi}{\partial A_i(t)}=&\frac{\partial}{\partial A_i(t)}\left(\frac{\partial p_i}{\partial A_i(t)}q^*\right)-(\rho+\delta)\frac{\partial}{\partial k_i(t)}\left(\frac{\gamma'(k^*)}{\frac{\partial \Gamma_i}{\partial k_i(t)}}\right)\frac{\delta}{\frac{\partial \Gamma_i}{\partial k_i(t)}+(n-1)\frac{\partial \Gamma_i}{\partial K_{-i}(t)}}\\
=&\frac{\partial}{\partial A_i(t)}\left(\frac{\partial p_i}{\partial A_i(t)}q^*\right)-\delta(\rho+\delta)\frac{\gamma''\frac{\partial \Gamma_i}{\partial k_i(t)}-\gamma'\left[\frac{\partial^2 \Gamma_i}{\partial k_i(t)^2}+(n-1)\frac{\partial^2 \Gamma_i}{\partial k_i(t)\partial K_{-i}(t)}\right]}{\left(\frac{\partial \Gamma_i}{\partial k_i(t)}\right)^2\left[\frac{\partial \Gamma_i}{\partial k_i(t)}+(n-1)\frac{\partial \Gamma_i}{\partial K_{-i}(t)}\right]}.
\end{align*}
From the assumptions (\ref{as0}) and (\ref{as2})
\[\gamma''\frac{\partial \Gamma_i}{\partial k_i(t)}-\gamma'\left[\frac{\partial^2 \Gamma_i}{\partial k_i(t)^2}+(n-1)\frac{\partial^2 \Gamma_i}{\partial k_i(t)\partial K_{-i}(t)}\right]>0.\]
Therefore, $\det{\Omega}<0$ means $\frac{\partial \Phi}{\partial A_i(t)}<0$, and we have shown the following result.
\begin{lemma}
The left-hand side of (\ref{ga1}), $\Phi$, is decreasing with respect to $A_i(t)$.
\end{lemma}

\subsection*{\textbf{Linear and quadratic example}}

Assume that the inverse demand function is
\[p_i(t)=A_i(t)-Bq_i(t)-D\sum_{j\neq i}q_j(t).\]
The advertising cost is
\[\gamma(k_i(t))=\frac{\alpha}{2}(k_i(t))^2.\]
The production cost is
\[c(q_i(t))=cq_i(t),\]
and the accumulation of advertising effects is written as 
\[\Gamma(k_i(t),K_{-i}(t))=k_i(t)+\beta\sum_{j\neq i}k_j(t),\ 0<\beta <1.\]
At the steady state, (\ref{q}) is reduced to
\[A^*-[B+D(n-1)]q^*-Bq^*-c=0.\]
Thus,
\[q^*=\frac{A^*-c}{2B+D(n-1)}.\]
(\ref{ga1}) is reduced to
\[q^*=(\rho+\delta)\alpha k^*.\]
Thus,
\begin{equation}
\frac{A^*-c}{2B+D(n-1)}=(\rho+\delta)\alpha k^*.\label{op1}
\end{equation}
From (\ref{st1})
\begin{equation}
A^*=\frac{\Gamma_i}{\delta}=\frac{[1+(n-1)\beta]k^*}{\delta}.\label{op2}
\end{equation}
Solving (\ref{op1}) and (\ref{op2}) yields
\[A^*=\frac{[1+(n-1)\beta]c}{1+(n-1)\beta-\alpha \delta(\rho+\delta)[2B+(n-1)D]},\]
and
\[k^*=\frac{c\delta}{1+(n-1)\beta-\alpha \delta(\rho+\delta)[2B+(n-1)D]}.\]
These are the results in \cite{cl6}.


\section{Advertising in a dynamic oligopoly: Memoryless closed-loop solution without spillover}

We seek to the solution of the memoryless closed-loop approach. As discussed in Introduction memoryless closed-loop and open-loop are not equivalent. For simplicity, we assume 
\begin{equation}
\frac{\partial \Gamma_i}{\partial K_{-i}(t)}=0,\label{as1}
\end{equation}
that is, there is no spillover effect of advertising investment. The first order conditions and the second order conditions for Firm $i$ are the same as those in the open-loop case. Using (\ref{as1}), they are
\begin{equation}
\frac{\partial \hat{\mathcal{H}}_i}{\partial q_i(t)}=p_i+\frac{\partial p_i}{\partial q_i(t)}q_i(t)-c'(q_i(t))=0,\label{cq1}
\end{equation}
\begin{equation}
\frac{\partial \hat{\mathcal{H}}_i}{\partial k_i(t)}=-\gamma'(k_i(t))+\lambda_{ii}(t)\frac{\partial \Gamma_i}{\partial k_i(t)}=0,\label{ck1}
\end{equation}
\[\frac{\partial \hat{\mathcal{H}}^2_i}{\partial q_i(t)^2}=2\frac{\partial p_i}{\partial q_i(t)}+\frac{\partial^2 p_i}{\partial q_i(t)^2}q_i(t)-c''(q_i(t))<0,\]
and
\[\frac{\partial^2 \hat{\mathcal{H}}_i}{\partial k_i(t)^2}=-\gamma''(k_i(t))+\lambda_{ii}(t)\frac{\partial^2 \Gamma_i}{\partial k_i(t)^2}<0,\ j\neq i.\]

The adjoint conditions are different from those in the open-loop case. They are
\begin{align}
&-\frac{\partial \hat{\mathcal{H}}_i}{\partial A_i(t)}-\sum_{j\neq i}\frac{\partial \hat{\mathcal{H}}_i}{\partial k_j(t)}\frac{\partial k_j(t)}{\partial A_i(t)}-\sum_{j\neq i}\frac{\partial \hat{\mathcal{H}}_i}{\partial q_j(t)}\frac{\partial q_j(t)}{\partial A_i(t)}\label{a11}\\
=&\frac{\partial \lambda_{ii}(t)}{\partial t}-\rho \lambda_{ii}(t),\ i\in \{1, 2, \dots,n\}, \notag
\end{align}
and
\begin{align}
&-\frac{\partial \hat{\mathcal{H}}_i}{\partial A_j(t)}-\sum_{l\neq i,j}\frac{\partial \hat{\mathcal{H}}_i}{\partial k_l(t)}\frac{\partial k_l(t)}{\partial A_j(t)}-\frac{\partial \hat{\mathcal{H}}_i}{\partial k_j(t)}\frac{\partial k_j(t)}{\partial A_j(t)}\label{a12}\\
&-\sum_{l\neq i,j}\frac{\partial \hat{\mathcal{H}}_i}{\partial q_l(t)}\frac{\partial q_l(t)}{\partial A_j(t)}-\frac{\partial \hat{\mathcal{H}}_i}{\partial q_j(t)}\frac{\partial q_j(t)}{\partial A_j(t)}=\frac{\partial \lambda_{ij}(t)}{\partial t}-\rho \lambda_{ij}(t),\ j\neq i.\notag
\end{align}
The terms in (\ref{a11})
\[-\sum_{j\neq i}\frac{\partial \hat{\mathcal{H}}_i}{\partial k_j(t)}\frac{\partial k_j(t)}{\partial A_i(t)}-\sum_{j\neq i}\frac{\partial \hat{\mathcal{H}}_i}{\partial q_j(t)}\frac{\partial q_j(t)}{\partial A_i(t)},\]
and
the terms in (\ref{a12})
\[-\sum_{l\neq i,j}\frac{\partial \hat{\mathcal{H}}_i}{\partial k_l(t)}\frac{\partial k_l(t)}{\partial A_j(t)}-\frac{\partial \hat{\mathcal{H}}_i}{\partial k_j(t)}\frac{\partial k_j(t)}{\partial A_j(t)}-\sum_{l\neq i,j}\frac{\partial \hat{\mathcal{H}}_i}{\partial q_l(t)}\frac{\partial q_l(t)}{\partial A_j(t)}-\frac{\partial \hat{\mathcal{H}}_i}{\partial q_j(t)}\frac{\partial q_j(t)}{\partial A_j(t)}\]
take into account the interaction between the control variables of the firms other than Firm $i$ and the current levels of the state variables.

We have
\begin{equation}
\frac{\partial \hat{\mathcal{H}}_i}{\partial A_i(t)}=\frac{\partial p_i}{\partial A_i(t)}q_i(t)-\delta\lambda_{ii}(t),\label{b11}
\end{equation}
\begin{equation*}
\frac{\partial \hat{\mathcal{H}}_i}{\partial A_j(t)}=-\delta\lambda_{ij}(t),\ j\neq i,
\end{equation*}
\begin{equation*}
\frac{\partial \hat{\mathcal{H}}_i}{\partial k_j(t)}=-\lambda_{ij}(t)\frac{\partial \Gamma_j}{\partial k_j(t)},\ j\neq i,
\end{equation*}
\begin{equation}
\frac{\partial \hat{\mathcal{H}}_i}{\partial q_j(t)}=\frac{\partial p_i}{\partial q_j(t)}q_i(t),\ j\neq i,\label{b31}
\end{equation}
and
\begin{equation*}
\frac{\partial k_j(t)}{\partial A_i(t)}=0,\ j\neq i.
\end{equation*}
$\frac{\partial q_j(t)}{\partial A_i(t)}$ is obtained by (\ref{2-ap3}) in Appendix. If the outputs of the firms are strategic substitutes $\left(\frac{\partial p_j}{\partial q_i(t)}+\frac{\partial^2 p_j}{\partial q_i(t)\partial q_j(t)}q_j(t)<0\right)$, $\frac{\partial q_j(t)}{\partial A_i(t)}>0$, and if they are strategic complements $\left(\frac{\partial p_j}{\partial q_i(t)}+\frac{\partial^2 p_j}{\partial q_i(t)\partial q_j(t)}q_j(t)>0\right)$, $\frac{\partial q_j(t)}{\partial A_i(t)}>0$.

At the steady state we have
\[\frac{dA_i(t)}{dt}=\Gamma(k_i(t), K_{-i}(t))-\delta A_i(t)=0,\]
and
\[\frac{\partial \lambda_{ii}}{\partial t}=0,\ \frac{\partial \lambda_{ij}}{\partial t}=0,\ i\in \{1, 2, \dots, n\},\ j\neq i.\]
By symmetry of the oligopoly we can assume $\lambda_{ii}(t)=\lambda_{jj}(t)$ for $j\neq i$, $\lambda_{ij}(t)=\lambda_{il}(t)=\lambda_{ji}(t)$ for $j, l\neq i$, $\frac{\partial k_j(t)}{\partial A_i(t)}=\frac{\partial k_l(t)}{\partial A_j(t)}$ for $l\neq j$, $A_i(t)=A_j(t)$, $k_j(t)=k_i(t)$, $q_j(t)=q_i(t)$ for $j\neq i$, and so on. Denote the steady state values of $\lambda_{ii}$, $\lambda_{ij}$, $A_i(t)$, $q_i(t)$ and $k_i(t)$ by $\lambda_{own}$, $\lambda_{other}$, $A^{**}$, $q^{**}$ and $k^{**}$. Then, with (\ref{b11}) and (\ref{b31}), (\ref{a11}) is rewritten as
\begin{equation}
\frac{\partial p_i}{\partial A_i(t)}q^*+(n-1)\frac{\partial p_i}{\partial q_j(t)}q^*\frac{\partial q_j(t)}{\partial A_i(t)}=(\rho+\delta) \lambda_{own}.\label{a110}
\end{equation}
The first order condition for the choice of $k_i(t)$, (\ref{ck1}), is reduced to
\[-\gamma'(k^{**})+\lambda_{own}\frac{\partial \Gamma_i}{\partial k_i(t)}=0.\]
\textcolor{black}{From (\ref{a110})} this means
\begin{align}
\left[\frac{\partial p_i}{\partial A_i(t)}+(n-1)\frac{\partial p_i}{\partial q_j(t)}\frac{\partial q_j(t)}{\partial A_i(t)}\right]q^*-(\rho+\delta)\frac{\gamma'(k^{**})}{\frac{\partial \Gamma_i}{\partial k_i(t)}}=0.\label{ga2}\end{align}

\textcolor{black}{From (\ref{m1}) and (\ref{a11} with (\ref{b11}), (\ref{b31})) we obtain
\begin{equation}
\frac{dA_i(t)}{dt}=\Gamma(k_i(t),(n-1)k_i(t))-\delta A_i(t),\label{d11}
\end{equation}
and
\begin{equation}
\frac{\partial \lambda_{ii}(t)}{\partial t}=-\frac{\partial p_i}{\partial A_i(t)}q_i(t)-(n-1)\frac{\partial p_i}{\partial q_j(t)}\frac{\partial q_j(t)}{\partial A_i(t)}q_i(t)+(\rho+\delta) \lambda_{ii}(t).\label{d21}
\end{equation}
Under the assumption $\frac{\partial \Gamma_i}{\partial K_{-i}(t)}=0$,} for the steady state to be a saddle point we need
\[\det{\Omega'}<0,\]
where
\[
\Omega'=
\begin{bmatrix}
-\delta&\textcolor{black}{\frac{\partial \Gamma_i}{\partial k_i(t)}\frac{\left(\frac{\partial \Gamma_i}{\partial k_i(t)}\right)^2}{\gamma''\frac{\partial \Gamma_i}{\partial k_i(t)}-\gamma'\frac{\partial^2 \Gamma_i}{\partial k_i(t)^2}}}\\
-\frac{\partial }{\partial A_i(t)}\left[\frac{\partial p_i}{\partial A_i(t)}q_i(t)+(n-1)\frac{\partial p_i}{\partial q_j(t)}\frac{\partial q_j(t)}{\partial A_i(t)}q_i(t)\right]&\rho+\delta
\end{bmatrix}.
\]
We assume $\det{\Omega'}<0$. This system of dynamics are obtained from (\ref{d11}) and (\ref{d21}).

The first order condition for the output choice in the memoryless closed-loop case, (\ref{cq1}), is the same as that, (\ref{q}), in the open-loop case. Thus, we have $q^{**}=q^*$ \emph{given} the value of $A_i(t)$. 



We show the following proposition.
\begin{proposition}
Assume that there is no spillover effect of advertising investment.

\begin{enumerate}
	\item If the outputs of the firms are strategic substitutes, the steady state value of $A_i(t)$ in the closed-loop solution is larger than that in the open-loop solution.
	\item If the outputs of the firms are strategic complements, the steady state value of $A_i(t)$ in the closed-loop solution is smaller than that in the open-loop solution.
\end{enumerate}
\end{proposition}
\begin{proof}
\begin{enumerate}
\item Let us compare (\ref{ga2}) with (\ref{ga1}) in the open-loop case.  Suppose that $A^{*}=A^{**}$. Then, $q^{*}=q^{**}$, $k^{*}=k^{**}$. They satisfy (\ref{ga2}), and the left-hand side of (\ref{ga1}), which is denoted by $\Phi$, is negative because $\frac{\partial p_i}{\partial q_j(t)}<0$ and $\frac{\partial q_j(t)}{\partial A_i(t)}<0$. Since $\Phi$ is decreasing with respect to $A_i(t)$, the steady state value of $A_i(t)$ in the open-loop solution is smaller than that in the closed-loop solution, or the steady state value of $A_i(t)$ in the closed-loop solution is larger than that in the open-loop solution.
\item If the outputs of the firms are strategic complements, $\frac{\partial q_j(t)}{\partial A_i(t)}<0$, and the left-hand side of (\ref{ga1}) is positive. Therefore, the steady state value of $A_i(t)$ in the closed-loop solution is smaller than that in the open-loop solution.
\end{enumerate}\qed
\end{proof}

\section{Advertising in a dynamic oligopoly: Feedback solution  without spillover}

We consider a solution of feedback approach using the HJB equation. Similarly to the previous section we assume 
\[\frac{\partial \Gamma_i}{\partial K_{-i}(t)}=0.\]
Let $V_i(A_1(t),A_2(t),\dots,A_n(t))$ be the value function of Firm $i,\ i\in \{1, 2, \dots, n\}$. The HJB equation for Firm $i$ is written as
\begin{align}
&\rho V_i(A_1(t),A_2(t),\dots,A_n(t))=\max_{q_i(t),k_i(t)}\{p_iq_i(t)-c(q_i(t))-\gamma_i(k_i(t))\label{hjb}\\
&+\frac{\partial V_i(A_1(t),A_2(t),\dots,A_n(t)}{\partial A_i(t)}(\Gamma_i-\delta A_i(t))+\sum_{j\neq i}\frac{\partial V_i(A_1(t),A_2(t),\dots,A_n(t)}{\partial A_j(t)}(\Gamma_j-\delta A_i(t))\}.\notag
\end{align}
The first order conditions are
\begin{equation}
p_i+\frac{\partial p_i}{\partial q_i(t)}q_i(t)-c'(q_i(t))=0,\label{f1}
\end{equation}
and
\begin{equation}
-\gamma'(k_i(t))+\frac{\partial V_i(A_1(t),A_2(t),\dots,A_n(t))}{\partial A_i(t)}\frac{\partial \Gamma_i}{\partial k_i(t)}=0.\label{f12}
\end{equation}
(\ref{f1}) is the same as (\ref{cq1}) in the memoryless closed-loop case. From (\ref{f12})
\begin{equation}
\frac{\partial V_i(A_1(t),A_2(t),\dots,A_n(t))}{\partial A_i(t)}=\frac{\gamma'(k_i(t))}{\frac{\partial \Gamma_i}{\partial k_i(t)}}.\label{f2}
\end{equation}
Substituting this into (\ref{hjb}), using symmetry, yields
\begin{align*}
\rho &V_i(A_1(t),A_2(t),\dots,A_n(t))=p_iq_i(t)-c(q_i(t))-\gamma_i(k_i(t)) \\
&+\frac{\gamma'(k_i(t))}{\frac{\partial \Gamma_i}{\partial k_i(t)}}[\Gamma_i-\delta A_i(t)]+(n-1)\frac{\partial V_i(A_1(t),A_2(t),\dots,A_n(t)}{\partial A_j(t)}[\Gamma_j-\delta A_j(t)],\ j\neq i.\notag
\end{align*}
This is an identity. Differentiating it with respect to $A_i(t)$ yields
\begin{align*}
\rho &\frac{\partial V_i(A_1(t),A_2(t),\dots,A_n(t))}{\partial A_i(t)}=\frac{\partial p_i}{\partial A_i(t)}q_i(t)+\left[p_i+\frac{\partial p_i}{\partial q_i(t)}q_i(t)-c'(q_i(t))\right]\frac{\partial q_i(t)}{\partial A_i(t)}\\
&+(n-1)\frac{\partial p_i}{\partial q_j(t)}q_i(t)\frac{\partial q_j(t)}{\partial A_i(t)}-\frac{\partial}{\partial k_i(t)}\left(\frac{\gamma'(k_i(t))}{\frac{\partial \Gamma_i}{\partial k_i(t)}}\right)\frac{\partial k_i(t)}{\partial A_i(t)}(\Gamma_i-\delta A_i(t))-\delta\frac{\gamma'(k_i(t))}{\frac{\partial \Gamma_i}{\partial k_i(t)}}\\
&+(n-1)\frac{\partial^2 V_i(A_1(t),A_2(t),\dots,A_n(t)}{\partial A_i(t)\partial A_j(t)}(\Gamma_j-\delta A_j(t))).
\end{align*}
At the steady state $\Gamma_i-\delta A_i(t)=\Gamma_j-\delta A_j(t)=0$. Thus, using (\ref{f1}), we get
\begin{align*}
\rho &\frac{\partial V_i(A_1(t),A_2(t),\dots,A_n(t))}{\partial A_i(t)}=\frac{\partial p_i}{\partial A_i(t)}q_i(t)+(n-1)\frac{\partial p_i}{\partial q_j(t)}q_i(t)\frac{\partial q_j(t)}{\partial A_i(t)}-\delta\frac{\gamma'(k_i(t))}{\frac{\partial \Gamma_i}{\partial k_i(t)}}.
\end{align*}
From this and (\ref{f2}),
\begin{align*}
\left[\frac{\partial p_i}{\partial A_i(t)}+(n-1)\frac{\partial p_i}{\partial q_j(t)}\frac{\partial q_j(t)}{\partial A_i(t)}\right]q_i(t)-(\rho+\delta)\frac{\gamma'(k_i(t))}{\frac{\partial \Gamma_i}{\partial k_i(t)}}=0.
\end{align*}
This is the same as (\ref{ga2}). Therefore, 
\begin{proposition}
If there is no  spillover effect of the advertising investment, the memoryless closed-loop solution and the feedback solution are equivalent.
\end{proposition}

\section{Cartel}

We consider full cartelization in which the firms cooperatively determine their outputs and advertising investments to maximize the discounted total profits. By symmetry we assume $p_i(t)=p_j(t)$, $q_i(t)=q_j(t))$, $k_i(t)=k_j(t)$, $A_i(t)=A_j(t)$ for any $j\neq i$. Denote the values of them by $p(t)$, $q(t)$, $k(t)$, $A(t)$. The objective of the firms is
\[\max_{q(t),k(t)}n\int_0^{\infty}n[p(A(t),q(t),\dots,q(t))q(t)-c(q(t))-\gamma(k(t))]dt.\]
subject to
\[\frac{dA(t)}{dt}=\Gamma(k(t),(n-1)k(t))-\delta A(t).\]
Let $\hat{\mathcal{H}}$ be the current value Hamiltonian function, then
\begin{align}
\hat{\mathcal{H}}=&n\left\{p(A(t),q(t),\dots,q(t))q(t)-c(q(t))-\gamma(k(t))\right\}\label{ha}\\
&+\lambda(t)\left[\Gamma(k(t),(n-1)k(t))-\delta A(t)\right].\notag
\end{align}
Let
\[\mu(t)=e^{-\rho t}\lambda(t).\]
$\mu(t)$ is the costate variable. Denote $p(A(t),q(t),\dots,q(t))$ by $p$. The first order conditions are
\begin{equation}
\frac{\partial \hat{\mathcal{H}}}{\partial q(t)}=n\left\{p+\left[\frac{\partial p_i}{\partial q_i(t)}+(n-1)\frac{\partial p_i}{\partial q_j(t)}\right]q(t)-c'(q(t))\right\}=0,\label{car1}
\end{equation}
and
\begin{equation}
\frac{\partial \hat{\mathcal{H}}}{\partial k(t)}=-n\gamma'+\lambda(t)\left[\frac{\partial \Gamma_i}{\partial k_i(t)}+(n-1)\frac{\partial \Gamma_i}{\partial K_{-i}(t)}\right]=0.\label{car2}
\end{equation}
The adjoint condition is
\begin{equation}
-\frac{\partial \hat{\mathcal{H}}}{\partial A(t)}=-n\frac{\partial p_i}{\partial A(t)}q_i(t)+\delta \lambda(t)=\frac{\partial \lambda(t)}{\partial t}-\rho \lambda(t).\label{car21}
\end{equation}
At the steady state
\[\frac{dA(t)}{dt}=\Gamma(k(t),(n-1)k(t))-\delta A(t)=0,\]
\[\frac{\partial \lambda(t)}{\partial t}=0.\]
Denote the steady state values of $A(t)$, $\lambda(t)$, $q(t)$ and $k(t)$ by $A^c$, $\lambda^c$, $q^c$ and $k^c$. From (\ref{car21})
\[n\frac{\partial p_i}{\partial A(t)}q^c-\lambda^c(\rho+\delta)=0.\]
Then, from (\ref{car2}) we have
\begin{equation}
\frac{\partial p_i}{\partial A(t)}q^c\left[\frac{\partial \Gamma_i}{\partial k_i(t)}+(n-1)\frac{\partial \Gamma_i}{\partial K_{-i}(t)}\right]-\left(\rho+\delta\right)\gamma'=0.\label{car22}
\end{equation}

We show the following proposition.
\begin{proposition}
If there is no spillover effect, the value of $A(t)$ in the cartel is smaller than the value of $A_i(t)$ in the open-loop solution.  
\end{proposition}
\begin{proof}
\begin{enumerate}
\item Compare (\ref{car1}) and the first order condition in the open-loop case (\ref{q}). If $q_i(t)$ satisfies (\ref{car1}), the left-hand side of (\ref{q}) is
\[-(n-1)\frac{\partial p_i}{\partial q_j(t)}q_i(t)>0.\]
Therefore, under the second order condition for the output choice in the open-loop case, the output of each firm in the open-loop solution is larger than that in the cartel, or the output of each firm in the cartel is smaller than that in the open-loop solution, that is, $q^c<q^*$ \emph{given $A(t)$} .

\item Assume $\frac{\partial \Gamma_i}{\partial K_{-i}(t)}=0$. Then, (\ref{car22}) is reduced to
\begin{equation}
\frac{\partial p_i}{\partial A(t)}q^c-\left(\rho+\delta\right)\frac{\gamma'}{\frac{\partial \Gamma_i}{\partial k_i(t)}}=0.\label{car23}
\end{equation}
Suppose that $A(t)$ and $q^c$, which is obtained from (\ref{car1}) with $A(t)$, satisfy (\ref{car23}). Then, the left-hand side of (\ref{ga1}) in the open-loop case, which is denoted by $\Phi$, is negative because by (\ref{aq}) $\frac{\partial p_i}{\partial A(t)}q_i(t)$ is deceasing with respect to $q_i(t)$ and $q^*>q^c$. Since $\Phi$ is decreasing with respect to $A_i(t)$, the value of $A_i(t)$ in the open-loop solution is larger than the value of $A(t)$  in the cartel, or the value of $A(t)$ in the cartel is smaller than the value of $A_i(t)$ in the open-loop solution. \qed
\end{enumerate}
\end{proof}
If there is spillover effect of advertising activities, the value of $A(t)$ in the cartel may be larger than $A_i(t)$ in the open-loop solution.


\subsection*{\textbf{Linear and quadratic example}}

Assume that the inverse demand function is
\[p_i(t)=A_i(t)-Bq_i(t)-D\sum_{j\neq i}q_j(t).\]
The advertising cost is
\[\gamma(k_i(t))=\frac{\alpha}{2}(k_i(t))^2.\]
The production cost is
\[c(q_i(t))=cq_i(t),\]
and the accumulation of advertising effects is written as 
\[\Gamma(k_i(t),K_{-i}(t))=k_i(t)+\beta\sum_{j\neq i}k_j(t).\]
At the steady state (\ref{car1}) and (\ref{car2}) are reduced to
\begin{equation}
A-[B+D(n-1)]q-[B+(n-1)D]q-cq=0,\label{car3}
\end{equation}
and
\begin{equation}
-\alpha nk+\lambda[1+(n-1)\beta]=0.\label{car4}
\end{equation}
From (\ref{car21}) with $\frac{\partial \lambda(t)}{\partial t}=0$,
\begin{equation}
(\rho+\delta)\lambda=nq.\label{car5}
\end{equation}
From (\ref{car3})
\begin{equation}
q=\frac{A-c}{2[B+(n-1)D]}.\label{car6}
\end{equation}
From $\frac{dA(t)}{dt}=k+(n-1)\beta k-\delta A=0$,
\begin{equation}
k=\frac{\delta A}{1+(n-1)\beta}.\label{car7}
\end{equation}
Solving (\ref{car4}), (\ref{car5}), (\ref{car6}) and (\ref{car7}),
\[A=\frac{[1+(n-1)\beta]^2c}{(1+(n-1)\beta)^2-2\alpha\delta(\rho+\delta)(B+(n-1)D)},\]
\[k=\frac{[1+(n-1)\beta]c\delta}{(1+(n-1)\beta)^2-2\alpha\delta(\rho+\delta)(B+(n-1)D)}.\]

In \cite{cl6} the current value Hamiltonian function in (\ref{ha}) is written as follows.
\[\hat{\mathcal{H}}=n\left\{p(A(t),q(t),\dots,q(t))q(t)-c(q(t))-\gamma(k(t))+\lambda(t)\left[\Gamma(k(t),(n-1)k(t))-\delta A(t)\right]\right\}.\]
Then, we obtain
\[A=\frac{[1+(n-1)\beta]^2nc}{n[1+(n-1)\beta]^2-2\alpha\delta(\rho+n\delta)[B+(n-1)D]},\]
\[k=\frac{[1+(n-1)\beta]nc\delta}{n[1+(n-1)\beta]^2-2\alpha\delta(\rho+n\delta)[B+(n-1)D]}.\]
These are the results in \cite{cl6}.



\section{Concluding Remark}

We have studied the problem of advertising activities in a dynamic oligopoly under general demand and cost functions by a differential game approach. We have shown that the results in \cite{cl6} in a case of linear demand and quadratic cost functions can be generalized to a case of general demand and cost functions, but some results depend on the property of demand functions, strategic substitutability or strategic complementartity.

\section*{Acknowledgment}

This work was supported by Japan Society for the Promotion of Science KAKENHI Grant Number 18K01594 and 18K12780.
\section*{\boldmath  Appendix: Derivation of $\frac{\partial q_j(t)}{\partial A_i(t)}$.}

Suppose a state such that $q_1(t)=q_2(t)=\dots=q_n(t)$. Denote $p_i(A_i(t),q_1(t), q_1(t), \dots, q_n(t))$ by $p_i$. The first order conditions for Firm $i$ and Firm $j,\ j\neq i$, are
\begin{equation*}
p_i+\frac{\partial p_i}{\partial q_i(t)}q_i(t)-c'(q_i(t))=0,
\end{equation*}
and
\begin{equation*}
p_j+\frac{\partial p_j}{\partial q_j(t)}q_j(t)-c'(q_j(t))=0,
\end{equation*}
Differentiating them with respect to $A_i(t)$ yields
\begin{align*}
&\left[2\frac{\partial p_i}{\partial q_i(t)}+\frac{\partial^2 p_i}{\partial q_i(t)^2}q_i(t)-c''(q_i(t))\right]\frac{\partial q_i(t)}{\partial A_i(t)}+(n-1)\left[\frac{\partial p_i}{\partial q_j(t)}+\frac{\partial^2 p_i}{\partial q_i(t)\partial q_j(t)}q_i(t)\right]\frac{\partial q_j(t)}{\partial A_i(t)}\\
=&-\left[\frac{\partial p_i}{\partial A_i(t)}+\frac{\partial^2 p_i}{\partial q_i(t)\partial A_i(t)}q_i(t)\right],
\end{align*}
and
\[\left[\frac{\partial p_j}{\partial q_i(t)}+\frac{\partial^2 p_j}{\partial q_i(t)\partial q_j(t)}q_j(t)\right]\frac{\partial q_i(t)}{\partial A_i(t)}+\left[n\frac{\partial p_j}{\partial q_j(t)}+(n-1)\frac{\partial^2 p_j}{\partial q_j(t)^2}q_j(t)-c''(q_j(t))\right]\frac{\partial q_j(t)}{\partial A_i(t)}=0.\]
Let
\[\Psi=-\left[\frac{\partial p_i}{\partial A_i(t)}+\frac{\partial^2 p_i}{\partial q_i(t)\partial A_i(t)}q_i(t)\right]<0.\]
From them we obtain
\[\frac{\partial q_i(t)}{\partial A_i(t)}=\frac{n\frac{\partial p_j}{\partial q_j(t)}+(n-1)\frac{\partial^2 p_j}{\partial q_j(t)^2}q_j(t)-c''(q_j(t))}{\Delta}\Psi>0,\]
and
\begin{equation}
\frac{\partial q_j(t)}{\partial A_i(t)}=-\frac{\frac{\partial p_j}{\partial q_i(t)}+\frac{\partial^2 p_j}{\partial q_i(t)\partial q_j(t)}q_j(t)}{\Delta}\Psi,\label{2-ap3}
\end{equation}
where
\begin{align*}
\Delta=&\left[2\frac{\partial p_i}{\partial q_i(t)}+\frac{\partial^2 p_i}{\partial q_i(t)^2}q_i(t)-c''(q_i(t))\right]\left[n\frac{\partial p_j}{\partial q_j(t)}+(n-1)\frac{\partial^2 p_j}{\partial q_j(t)^2}q_j(t)-c''(q_j(t))\right]\\
&-(n-1)\left[\frac{\partial p_i}{\partial q_j(t)}+\frac{\partial^2 p_i}{\partial q_i(t)\partial q_j(t)}q_i(t)\right]\left[\frac{\partial p_j}{\partial q_i(t)}+\frac{\partial^2 p_j}{\partial q_i(t)\partial q_j(t)}q_j(t)\right]>0.
\end{align*}
If the outputs of the firms are strategic substitutes $\left(\frac{\partial p_j}{\partial q_i(t)}+\frac{\partial^2 p_j}{\partial q_i(t)\partial q_j(t)}q_j(t)<0\right)$, $\frac{\partial q_j(t)}{\partial A_i(t)}<0$, and if they are strategic complements $\left(\frac{\partial p_j}{\partial q_i(t)}+\frac{\partial^2 p_j}{\partial q_i(t)\partial q_j(t)}q_j(t)>0\right)$, $\frac{\partial q_j(t)}{\partial A_i(t)}>0$.

\bibliographystyle{jecon} 
\bibliography{yatanaka}

\end{document}